\documentclass[12pt]{amsart}

\usepackage{a4wide}
\usepackage{amsmath, amssymb}
\usepackage{amsthm}
\usepackage[utf8]{inputenc}
\usepackage{subfigure}
\usepackage{enumerate}
\usepackage{cancel}
\usepackage{xspace}
\usepackage{float}

\usepackage{mathrsfs}
\usepackage[colorlinks, citecolor=blue, anchorcolor=blue, linkcolor=blue, urlcolor=blue]{hyperref}
\usepackage{color,xcolor}

\usepackage{wrapfig}

\usepackage{tikz}
\usetikzlibrary{patterns}

\usepackage{algorithm}
\usepackage{algorithmic}

\def\SYoung#1{\vbox{\smallskip\offinterlineskip
    \halign{&\vbox{##}\kern-\SThickness\cr #1}}}

\newdimen\SSquaresize \SSquaresize=4.5pt
\newdimen\SThickness \SThickness=.15pt
\newdimen\SCorrection \SCorrection=7pt

\def\SCarre#1{\hbox{\vrule width \SThickness
   \vbox to \SSquaresize{\hrule height \SThickness\vss
      \hbox to \SSquaresize{\hss$\scriptstyle#1$\hss}
   \vss\hrule height\SThickness}
   \unskip\vrule width \SThickness}
   \kern-\SThickness}

\allowdisplaybreaks

\makeatletter \@addtoreset{equation}{section}

\newtheorem{theorem}{Theorem}[section]
\newtheorem{proposition}[theorem]{Proposition}

\newtheorem{conjecture}[theorem]{Conjecture}

\newcounter{mmacnt}
\def\restartmma{\setcounter{mmacnt}{0}}
\restartmma
\catcode`|=\active
\def|#1|{\mathrm{#1}}
\catcode`|=12
\def\mma@slash{/}
\catcode`/=\active
\makeatother\def\mma{@}\makeatletter\let\mma@at\mma\let\mma\undefined
\def/{\@ifnextchar/{\ifmmode\mathop{\;\mma@slash\!\mma@slash\;}\else\mma@slash\!\mma@slash\fi\mma@eat}{\mma@B}}
\def\mma@B{\@ifnextchar.{\ifmmode\mathop{\;\mma@slash\!.\;}\else\mma@slash\!.\fi\mma@eat}{\mma@C}}
\def\mma@C{\expandafter\@ifnextchar\mma@at{\ifmmode\mathop{\;\mma@slash\kern-.25ex @\;}\else\mma@slash\!@\fi\mma@eat}{\mma@D}}
\def\mma@D{\@ifnextchar;{\ifmmode\mathop{\;\mma@slash\!;\;}\else\mma@slash\!;\fi\mma@eat}{\mma@slash}}
\catcode`/=12
\def\mma@eat#1{}
\newenvironment{mma}{%
 \par\smallskip
 \catcode`|=\active\catcode`/=\active
 \parskip=0pt\parindent=0pt % locally
 \small
 \def\In##1\\{%
   \def\linebreak{\hfil\break\null\qquad}%
   \refstepcounter{mmacnt}%
   \hangindent=2.5em\hangafter=0
   \leavevmode
   \llap{\tiny\sffamily In[\arabic{mmacnt}]:=\kern.5em}%
   \mathversion{bold}$\displaystyle##1$
   \mathversion{normal}\par\kern-\lineskip
 }%
 \def\Print##1\\{%
   \def\linebreak{\hfil\break}%
   \hangindent=2.5em\hangafter=0
   \leavevmode ##1\par}%
 \def\Out{\@ifnextchar[{\mma@out}{\mma@out[@]}}%
 \def\mma@out[##1]##2\\{%
   \def\linebreak{\hfil\break\null}%
   \kern\abovedisplayskip\par
   \hangindent=2.5em\hangafter=0
   {\kern\lineskip\advance\lineskip2pt\kern-\lineskip
   \leavevmode
   \llap{\tiny\sffamily Out[\arabic{mmacnt}]=\kern.5em}%
   \ifx ##1@\else
     \rlap{\kern\hsize\kern-2.5em\llap{(##1)}}%
   \fi
   $\displaystyle##2$\hfil\null
   \par\kern-\lineskip}%
   \kern\lineskip
   \kern\belowdisplayskip
 }%
 \def\Warning##1##2\\{%
   \def\linebreak{\hfil\break}%
   \hangindent=2.5em\hangafter=0
   \leavevmode
   {\scriptsize##1 : ##2}\par}%
}{%
}

\newcommand{\refOut}[1]{{\small\sffamily Out[#1]}}

\def\MLabel#1{{\refstepcounter{mmacnt}\label{#1}}\addtocounter{mmacnt}{-1}}

\title[$\infty$-log-concavity \& higher order Tur\'{a}n inequality for $\{g_{U_{n,d}}(t)\}_{d=1}^{n-1}$]{Infinite log-concavity and higher order Tur\'{a}n inequality for the sequences of Speyer's $g$-polynomial of uniform matroids}

\author[James J. Y. Zhao]{James Jing Yu Zhao}
       \address{School of Accounting, Guangzhou College of Technology and Business,
       Foshan 528138, P. R. China.}
       \email{zhao@gzgs.edu.cn}



\begin{document}

\begin{abstract}
Let $U_{n,d}$ be the uniform matroid of rank $d$ on $n$ elements. Denote by $g_{U_{n,d}}(t)$ the Speyer's $g$-polynomial of $U_{n,d}$. The Tur\'{a}n inequality and higher order Tur\'{a}n inequality are related to the Laguerre-P\'{o}lya ($\mathcal{L}$-$\mathcal{P}$) class of real entire functions, and the $\mathcal{L}$-$\mathcal{P}$ class has close relation with the Riemann hypothesis. The Tur\'{a}n type inequalities have received much attention. Infinite log-concavity is also a deep generalization of Tur\'{a}n inequality with different direction. In this paper, we mainly obtain the infinite log-concavity and the higher order Tur\'{a}n inequality of the sequence $\{g_{U_{n,d}}(t)\}_{d=1}^{n-1}$ for any $t>0$. In order to prove these results, we show that the generating function of $g_{U_{n,d}}(t)$, denoted $h_n(x;t)$, has only real zeros for $t>0$. Consequently, for $t>0$, we also obtain the $\gamma$-positivity of the polynomial $h_n(x;t)$, the asymptotical normality of $g_{U_{n,d}}(t)$, and the Laguerre inequalities for $g_{U_{n,d}}(t)$ and $h_n(x;t)$.
\end{abstract}

\maketitle

\section{Introduction}\label{sec:1}

In the study of the tropical analogues of the notions of linear spaces and Pl\"{u}cker coordinates, Speyer \cite{Speyer2008} found that all constructible tropical linear spaces have the same $f$-vector and are series-parallel. He also conjectured that this $f$-vector is maximal for all tropical linear spaces, with equality precisely for the series-parallel tropical linear spaces.
In the study of his $f$-vector conjecture, Speyer \cite{Speyer2009} discovered an invariant, denoted by $g_M(t)$,
for a (loopless and coloopless) $\mathbb{C}$-realizable matroid $M$.
By using this invariant, Speyer proved bounds on the complexity of Kapranov’s Lie
complexes \cite{Kapranov1993}, Hacking, Keel and Tevelev’s very stable pairs \cite{HKT2006}, and tropical linear spaces defined in \cite{Speyer2008}. It turns out that this invariant is of independent combinatorial interest.
In a subsequent work of Fink and Speyer \cite{FinkSpeyer2012}, the invariant $g_M(t)$ was defined for
an arbitrary (loopless and coloopless) matroid $M$.
L\'opez de Medrano, Rinc\'on and Shaw \cite{LRS2020} conjectured a Chow-theoretic formula for Speyer's $g$-polynomial, which was later proved by Berget, Eur, Spink and Tseng \cite{BEST-2023}. Ferroni \cite{Ferroni} provided a combinatorial way of computing $g_M(t)$ for an arbitrary Schubert matroid $M$. Recently, Ferroni and Schr\"{o}ter \cite{Ferroni-Schroter-2023} proved that the coefficients of $g_M(t)$ are nonnegative for any sparse paving matroid $M$.

It is worthwhile to mention that the $g$-polynomial of uniform matroids has been computed by Speyer \cite{Speyer2009}, and this polynomial plays a key role in his proof of the $f$-vector conjecture in the case of a tropical linear space realizable in characteristic zero. Denote by $U_{n,d}$ the uniform matroid of rank $d$ on $n$ elements. Speyer \cite{Speyer2009} obtained the following result.

\begin{proposition}[{\cite[Proposition 10.1]{Speyer2009}}]
For any $n\geq 2$ and $1\leq d\leq n-1$, the $g$-polynomial of the uniform matroid $U_{n,d}$ is given by
\begin{align}\label{g-uni-dn(t)}
 g_{U_{n,d}}(t)=\sum_{i=1}^{\min(d,n-d)} \frac{(n-i-1)!}{(d-i)!(n-d-i)!(i-1)!}\, t^i.
\end{align}
\end{proposition}

Recently, Zhang and Zhao \cite{Zhang-Zhao-2024} showed the real-rootedness of the polynomial $g_{U_{n,d}}(t)$ and asymptotic normality of the coefficient of $g_{U_{n,[n/2]}}(t)$. This motivated us to explore more properties of this polynomial.

The objective of this paper is to prove the infinite log-concavity and higher order Tur\'{a}n inequality of the sequence $\{g_{U_{n,d}}(t)\}_{d=1}^{n-1}$ for any given $t>0$ and $n\ge 4$ or $5$. Note that a polynomial $f(x)$ is usually called log-concave if its coefficient sequence is log-concave, and $f(x)$ is also called infinitely log-concave if its coefficient sequence enjoys this property. The main results obtained in the present paper is different from those in \cite{Zhang-Zhao-2024}. The real-rootedness of $g_{U_{n,d}}(t)$ in $t$ only implies the infinite log-concavity of the coefficient of $g_{U_{n,d}}(t)$. While in this paper, we show the infinite log-concavity of sequence $\{g_{U_{n,d}}(t)\}_{d=1}^{n-1}$.

A sequence $\{a_n\}_{n\geq 0}$ of real numbers is said to be \emph{log-concave} if $a_n^2-a_{n-1}a_{n+1}\geq 0$ for $n\geq 1$. This inequality is also called the \emph{Tur\'{a}n inequality} \cite{Szego1948, Turan} or the \emph{Newton's inequality} \cite{Craven-Csordas1989, CNV1986, Niculescu2000}.
A polynomial is said to be log-concave if the sequence of its coefficients is log-concave, see Brenti \cite{Brenti} and Stanley \cite{Stanley}.

For a sequence $\{a_n\}_{n \ge 0}$, define an operator $\mathfrak{L}$ by $\mathfrak{L}(\{a_n\}_{n\ge 0})=\{b_n\}_{n\ge 0}$, where $b_n=a_n^2-a_{n-1}a_{n+1}$ for $n\ge 0$, with the convention that $a_{-1}=0$. A sequence $\{a_n\}_{n\ge 0}$ is said to be $k$-log-concave if the sequence $\mathfrak{L}^j\left(\{a_n\}_{n\ge 0}\right)$ is nonnegative for each $1\le j\le k$, and $\{a_n\}_{n\ge 0}$ is said to be $\infty$-log-concave if $\mathfrak{L}^k\left(\{a_n\}_{n\ge 0}\right)$ is nonnegative for any $k\ge 1$.

The notion of infinite log-concavity was introduced by Boros and Moll \cite{Boros-Moll-2004}. They also conjectured that a sequence arose from their study of evaluation of a quartic integral, named $\{d_\ell(m)\}_{\ell=0}^m$, is infinitely log-concave, which is still open.
This conjecture has inspired many interesting works. For example,
Br\"{a}nd\'{e}n \cite{Branden2011} proved the following result and hence confirmed a conjecture independently posed by Fisk \cite{Fisk}, McNamara and Sagan \cite{McNamara-Sagan} and Stanley; see \cite{Branden2011}.
\begin{theorem}\cite{Branden2011}\label{thm:Branden}
If a polynomial $\sum_{k=0}^n a_k x^k$ has only real and negative zeros, then so does the polynomial $\sum_{k=0}^n (a_k^2-a_{k-1}a_{k+1}) z^k$ with the convention that $a_{-1}=a_{n+1}=0$. In particular, the sequence $\{a_k\}_{k=0}^n$ is infinitely log-concave.
\end{theorem}

The first main result of this paper is as follows.

\begin{theorem}\label{thm:g-poly-inflog}
Given $t>0$, the sequence $\{g_{U_{n,d}}(t)\}_{d=1}^{n-1}$ is infinitely log-concave for $n\ge 4$.
\end{theorem}

A real sequence $\{a_n\}_{n\geq 0}$ is said to satisfy the \emph{higher order Tur\'{a}n inequality} or cubic Newton inequality if for $n\geq 1$,
\begin{align}\label{eq:ho-Turan}
 4(a_n^2-a_{n-1}a_{n+1})(a_{n+1}^2-a_n a_{n+2})
-(a_{n} a_{n+1}-a_{n-1} a_{n+2})^2 \geq 0,
\end{align}
see \cite{Dimitrov1998, Niculescu2000, Rosset1989}.
The Tur\'{a}n inequality and the higher order Tur\'{a}n inequality are related to the Laguerre-P\'{o}lya class of real entire functions, see \cite{Dimitrov1998, Szego1948}.
A real entire function $\psi(x)=\sum_{n=0}^{\infty} a_n \frac{x^n}{n!}$ is said to belong to the Laguerre-P\'{o}lya class, denoted by $\psi \in \mathcal{L}$-$\mathcal{P}$, if
$\psi(x)=c x^m e^{-\alpha x^2+\beta x}\prod_{k=1}^{\infty}(1+x/x_k)e^{-x/x_k},$ where $c,\beta,x_k$ are real, $\alpha\geq 0$, $m$ is a nonnegative integer and $\sum x_k^{-2}<\infty$.
Jensen \cite{Jensen} proved that a real entire function $\psi(x)\in \mathcal{L}$-$\mathcal{P}$ if and only if for any integer $n>0$, the $n$-th associated Jenson polynomial
$
J_n(x)=\sum_{k=0}^n \binom{n}{k} a_k x^k
$
are hyperbolic, i.e., $J_n(x)$ has only real zeros.
This result was also obtained by P\'{o}lya and Schur \cite{Polya-Schur1933}.
Besides, if a real entire function $\psi \in \mathcal{L}$-$\mathcal{P}$, then its Maclaurin coefficients are log-concave, see \cite{Craven-Csordas1989, Csordas-Varga1990}.
Moreover, Dimitrov \cite{Dimitrov1998} proved that the higher order Tur\'{a}n inequality \eqref{eq:ho-Turan}, an extension of Tur\'{a}n inequality, is also a necessary condition for  $\psi \in \mathcal{L}{\textrm{-}}\mathcal{P}$.

The $\mathcal{L}$-$\mathcal{P}$ class has close relation with the Riemann hypothesis. Let $\zeta$ and $\Gamma$ denote the Riemann zeta function and the gamma function, respectively. The Riemann $\xi$-function is defined by
$
\xi(iz)=\frac{1}{2}\left(z^2-\frac{1}{4}\right)\pi^{-\frac{z}{2}-\frac{1}{4}}
  \Gamma\left(\frac{z}{2}+\frac{1}{4}\right)
  \zeta\left(z+\frac{1}{2}\right)
$,
see Boas \cite{Boas1954}.
It is well known that the Riemann $\xi$-function is an entire function of order $1$ and can be restated as
\begin{equation}\label{eq:R-xi-r}
 \frac{1}{8}\,\xi\!\left(\frac{x}{2}\right)
=\sum\limits_{k=0}^{\infty}(-1)^k\, \hat{b}_k\, \frac{x^{2k}}{(2k)!},
\end{equation}
where
$\hat{b}_k=\int_{0}^{\infty} t^{2t} \Phi(t) dt$
and
$\Phi(t)=\sum_{n=0}^\infty (2n^4 \pi^2 e^{9t} - 3n^2 \pi e^{5t}) \exp(-n^2 \pi e^{4t})$,
see P\'{o}lya \cite{Polya1927}. Set $z=-x^2$ in \eqref{eq:R-xi-r}. Then one obtain an entire function of order $1/2$, denoted by $\xi_1(z)$, that is,
\begin{equation*}
\xi_1(z)=\sum_{k=0}^\infty \frac{k!}{(2k)!}\hat{b}_k \frac{z^k}{k!}.
\end{equation*}
So the Riemann hypothesis holds if and only if $\xi_1(z)\in \mathcal{L}$-$\mathcal{P}$. See \cite{CJW2019, CNV1986, Dimitrov1998} for more details.

For the deep relation stated above, the higher order Tur\'{a}n inequality has received great attention. Many interesting sequences were showed to satisfy the higher order Tur\'{a}n inequality \eqref{eq:ho-Turan}. For instance, the Riemann $\xi$-function was proved to satisfy \eqref{eq:ho-Turan} by Dimitrov and Lucas \cite{DL2011}.
By using the Hardy-Ramanujan-Rademacher formula, Chen, Jia and Wang \cite{CJW2019} proved that the partition function $p(n)$ satisfies the higher order Tur\'{a}n inequality for $n\geq 95$ and the $3$-rd associated Jensen polynomials $\sum_{k=0}^3 \binom{3}{k} p(n+k)x^k$ have three distinct zeros, and hence confirm a conjecture of Chen \cite{Chen2017}.
Griffin, Ono, Rolen and Zagier \cite{GORZ2019} showed that Jensen polynomials for a large family of functions, including those associated to $\zeta(s)$ and the partition function $p(n)$, are hyperbolic for sufficiently large $n$, which gave evidence for Riemann hypothesis. For more results on this topic, please see \cite{Wang2019, Wagner2022, Wang-Zhang2024, Dong-Ji2024} and the literature cited therein.

The second main result of this section is as follows.

\begin{theorem}\label{thm:g-poly-hoTi}
Given $t>0$, the sequence $\{g_{U_{n,d}}(t)\}_{d=1}^{n-1}$ satisfies the higher order Tur\'{a}n inequality \eqref{eq:ho-Turan} for $n\ge 5$.
That is, for any $t>0$ and $n\ge 5$, we have
\begin{align}\label{ieq:h-T-gud}
 &\ 4\left(g_{U_{n,d}}(t)^2-g_{U_{n,d-1}}(t) g_{U_{n,d+1}}(t)\right)
      \left(g_{U_{n,d+1}}(t)^2-g_{U_{n,d}}(t) g_{U_{n,d+2}}(t)\right)\\
&\qquad \ge \left(g_{U_{n,d}}(t) g_{U_{n,d+1}}(t)-g_{U_{n,d-1}}(t) g_{U_{n,d+2}}(t)\right)^2,\nonumber
\end{align}
for $2\le d \le n-3$.
\end{theorem}

In order to prove Theorems \ref{thm:g-poly-inflog} and \ref{thm:g-poly-hoTi}, we investigate the real-rootedness of the following generating function of Speyer's $g$-polynomial
\begin{align}\label{eq:ex-g-poly-defi}
 h_n(x;t):
=\sum_{d=1}^{n-1}g_{U_{n,d}}(t) x^d,
\end{align}
where $t$ is treated as a nonzero parameter. Clearly, $h_n(x;t)$ is a polynomial in $x$ with leading term $tx^{n-1}$.
If $h_n(x;t)$ has only real zeros, then we can derive Theorem \ref{thm:g-poly-inflog} from Theorem \ref{thm:Branden} of Br\"{a}nd\'{e}n. As will be seen, the real-rootedness of $h_n(x;t)$ also plays a key role in the proof of Theorem \ref{thm:g-poly-hoTi}.

For notational convenience, let $g_{n,d}(t)=g_{U_{n,d}}(t)$ and $S_i(n,d)=[t^i]g_{n,d}(t)$, the coefficient of $t^i$ in $g_{n,d}(t)$. From \eqref{g-uni-dn(t)} it immediately follows that $g_{n,d}(t)= g_{n,n-d}(t)$.
So the polynomial $h_n(x;t)$ is symmetric for any $n\ge 3$.
Clearly,
$
S_i(n,d)=\frac{(n-i-1)!}{(d-i)!(n-d-i)!(i-1)!}=\binom{n-i-1}{d-i}\binom{n-d-1}{i-1}.
$
Since the binomial coefficient $\binom{n}{k}=0$ for $k>n$, we can rewrite \eqref{eq:ex-g-poly-defi} as
\begin{align}\label{eq:ex-g-poly-defi2}
 h_n(x;t):
=\sum_{d=1}^{n-1}\sum_{i=1}^{d}
  \binom{n-i-1}{d-i}\binom{n-d-1}{i-1}\, t^i x^d.
\end{align}

The remainder of this paper is organized as follows. In Section \ref{Sec:2}, we first obtain several recurrence relations of $h_n(x;t)$ by using a computer algebraic system  developed by Koutschan \cite{Koutschan2009}.
With one of these recurrences and a result due to Liu and Wang \cite{LiuWang}, we show that the polynomial sequence $\{h_n(x;t)\}_{n\ge 2}$ has interlacing zeros for $t>0$, in Section \ref{Sec:3}, and hence obtain the real-rootedness of $h_n(x;t)$ for $t>0$ and $n\ge 2$. In Section \ref{sec:proof}, we complete the proofs of Theorems \ref{thm:g-poly-inflog} and \ref{thm:g-poly-hoTi}.
The number $g_{n,d}(1)$ possesses specific geometric interpretations which will also be stated in detail in Section \ref{sec:5}.
It turns out that the polynomial $h_n(x;t)$ enjoys more properties. In Sectin \ref{sec:g-p}, we prove the $\gamma$-positivity of $h_n(x;t)$.
In Section \ref{Sec:asynor}, we show that the coefficient of $h_n(x;t)$, namely the Speyer's $g$-polynomial $g_{n,d}(t)$, is asymptotically normal by local and central limit theorems based on a criterion given by Bender \cite{Bender} and Harper \cite{Harper}.
Finally, in Section \ref{sec:L}, we give Laguerre inequalities related to $g_{n,d}(t)$ and $h_n(x;t)$ and conclude this paper with a conjecture.

\section{Recurrence Relations}\label{Sec:2}

The aim of this section is to show some recurrence relations for the generating function of Speyer's $g$-polynomial $h_n(x;t)$. As will be seen, these recurrences play an important role in the proofs of the real-rootedness, the $\gamma$-positivity, and the asymptotic normality presented in Sections \ref{Sec:3}, \ref{sec:g-p}, and \ref{Sec:asynor}, respectively. And the real-rootedness of $h_n(x;t)$ will lead to Theorems \ref{thm:g-poly-inflog} and \ref{thm:g-poly-hoTi}.
The main result of this section is as follows.

\begin{theorem}\label{rec:shiftn}
Let $h_n(x;t)$ be a function of $x$ defined by \eqref{eq:ex-g-poly-defi2}. Then for $n\ge 2$, we have
\begin{align}
 h_{n+2}(x;t)&\,
=(1+x)h_{n+1}(x;t)+txh_n(x;t),\label{rec:gfhx-n2}\\
 A\, h_{n+1}(x;t)&\,
=B\, h_{n}(x;t)+C\, h_{n}'(x;t),\label{rec:gfhx-nx'}\\
 \widetilde{A}\, h_{n+1}(x;t)&\,
=\widetilde{B}\, h_{n}(x;t)+C^2\, h_n''(x;t),\label{rec:gfhx-nx''}
\end{align}
where
\begin{align*}
 A = &\, (n-1)(x-1),\\
 B = &\, -n(2tx+x+1),\\
 C = &\, x(x^2+4tx+2x+1),
\end{align*}
and
\begin{align*}
 \widetilde{A}
= &\, (n-1)\big((n-2)x^3+(4nt+n-2t+2)x^2-(4nt+n-6t-4)x-n\big),\\
 \widetilde{B}
= &\, -n\big((nt-5t-2)x^3+(4nt^2+2nt-8t^2+n-6t-5)x^2\\
    &\qquad\qquad\qquad +(5nt+2n-9t-4)x+n-1\big).
\end{align*}
\end{theorem}

\begin{proof}
Since the coefficients of the recurrence \eqref{rec:gfhx-nx''} looks a little complicated, a symbolic proof based on a computer algebra system  will be more concise. We shall apply the package {\tt HolonomicFunctions\footnote{See https://www3.risc.jku.at/research/combinat/software/ergosum/RISC/HolonomicFunctions.html.}} for {\tt Mathematica} developed by Koutschan \cite{Koutschan2009} to prove Theorem \ref{rec:shiftn}. Let us first load this package.
\begin{mma}
\In  << |RISC'HolonomicFunctions'| \\
\Print HolonomicFunctions Package version 1.7.3 (21-Mar-2017)\\
\Print written by Christoph Koutschan\\
\Print Copyright Research Institute for Symbolic Computation (RISC),\\
\Print Johannes Kepler University, Linz, Austria\\
\end{mma}%

\noindent Then run the command Annihilator[$expr$, $ops$] which computes annihilating relations for $expr$ with respect to the Ore operator(s) $ops$. In the following input, $expr=h_n(x;t)$ and $ops$ consists of S$[n]$ and Der$[x]$, where S$[n]$ stands for the forward shift operator such that S$[n]h_n(x;t)=h_{n+1}(x;t)$ and Der$[x]$ represents the operator partial derivative with respect to $x$.
\begin{mma}\MLabel{MMA:9}
\In |ann = Annihilator|[|Sum|[|Sum|[|Binomial|[n - i - 1, d - i]|Binomial|[n - d - 1, i - 1] t^i,\linebreak \{i, 1,
      d\}] x^d, \{d, 1, n - 1\}], \{|S|[n], |Der|[x]\}]\\

\Out \{(1 - n - x + n x) S_n+(-x - 2 x^2 - 4 t x^2 - x^3) D_x + (n + n x + 2 n t x),\linebreak
    (x^2+x^3+4tx^3-x^4-4tx^4-x^5)D_x^2 + (-nx+4x^2-nx^2+6tx^2 - 4ntx^2 + 2x^3 + nx^3 - 2tx^3\linebreak
     + 4ntx^3 - 2x^4 + n x^4)D_x + (n - 3 n x + n^2 x - n t x + n^2 t x - n^2 x^2 - n t x^2 - n^2 t x^2)\}\\
\end{mma}

The output \refOut{\ref{MMA:9}} gives a list of Ore Polynomial expressions which form a Gr\"{o}bner basis of an annihilating left ideal for $h_n(x;t)$, where $S_n={\rm S}[n]$ and $D_{\!x}={\rm Der}[x]$. For more information of the Ore algebra and Ore polynomial, please see Ore \cite{Ore1933} and Koutschan \cite{Koutschan2009}.

Next, we employ the command OreReduce[$opoly$, $ann$] which reduces the Ore polynomial $opoly$ modulo the Gr\"{o}bner basis $ann$.
In the following input, $opoly:={\rm S}[n]^2 -(x+1){\rm S}[n] - t x$ which corresponds to the Ore polynomial of recurrence \eqref{rec:gfhx-n2} and $ann$ is given by \refOut{\ref{MMA:9}}.
\begin{mma}
\In |OreReduce|[|S|[n]^2 - (x+1) |S|[n] - t x, |ann|]\\

\Out 0\\
\end{mma}
The return is zero, which means that $opoly$ is contained in the left ideal generated by the elements of $ann$, and hence $\big({\rm S}[n]^2 -(x+1){\rm S}[n]  -t x\big)h_n(x;t)=0$, that is, the recurrence \eqref{rec:gfhx-n2} holds.

To prove \eqref{rec:gfhx-nx'}, set $opoly:=(n-1)(x-1){\rm S}[n] -x(x^2+4tx+2x+1){\rm Der}[x] + n(2tx+x+1)$ and run the following command:
\begin{mma}\MLabel{MMA:11}
\In |OreReduce|[(n - 1) (x - 1) |S|[n] - x (x^2 + 4 t x + 2 x + 1) |Der|[x] +
 n (2 t x + x + 1), |ann|]\\

\Out 0\\
\end{mma}
This means that \eqref{rec:gfhx-nx'} is valid.

Finally, by setting $opoly=\widetilde{A}\, {\rm S}[n]-\widetilde{B}-C^2{\rm Der}[x]^2$ where $\widetilde{A}$, $\widetilde{B}$, and $C$ are given in Theorem \ref{rec:shiftn}, we have
\begin{mma}\MLabel{MMA:7}
\In |OreReduce|[(n - 1) ((n - 2) x^3 + (4 n t + n - 2 t + 2) x^2 - (4 n t + n - 6 t - 4) x \linebreak
      - n) |S|[n] +
       n ((n t - 5 t - 2) x^3 + (4 n t^2 + 2 n t - 8 t^2 + n - 6 t - 5) x^2 \linebreak
      + (5 n t + 2 n - 9 t - 4) x + n - 1) -  x^2 (x^2 + 4 t x + 2 x + 1)^2 |Der|[x]^2, |ann|]\\

\Out 0\\
\end{mma}
The \refOut{\ref{MMA:7}} means that the recurrence \eqref{rec:gfhx-nx''} is true.
This completes the proof.
\end{proof}

\section{Real-rootedness and interlacing property}\label{Sec:3}

The objective of this section is to show the real-rootedness of the generating function of Speyer's $g$-polynomial $h_n(x;t)$ for any $t>0$.

\begin{theorem}\label{thm:rz-mfgnx}
For any $t>0$, the polynomial $h_n(x;t)$ in $x$ has only real zeros for all $n\ge 2$.
\end{theorem}

Before proving Theorem \ref{thm:rz-mfgnx}, let us first recall some definitions and results on Sturm sequences. Let PF be the set of real-rooted polynomials with nonnegative coefficients, including any nonnegative constant for convenience.
Given two polynomials $f(x),\,g(x)\in \mathrm{PF}$, assume $f(u_i)=0$ and $g(v_j)=0$.
We say that $g(x)$ interlaces $f(x)$, denoted $g(x)\preceq f(x)$, if either
$\deg f(x)=\deg g(x)=n$ and
\begin{align}\label{defi:inter-z1}
{v_n\le u_n\le\cdots\le v_2\le u_2\le v_1\le u_1},
\end{align}
or $\deg f(x)=\deg g(x)+1=n$ and
\begin{align}\label{defi:inter-z2}
{u_{n}\le v_{n-1}\le u_{n-1}\le\cdots\le v_{2}\le u_{2}\le v_{1}\le u_{1}}.
\end{align}
We say that $g(x)$ interlaces $f(x)$ strictly, denoted $g(x)\prec f(x)$, if the inequalities in \eqref{defi:inter-z1} or \eqref{defi:inter-z2} hold strictly.
Following Liu and Wang \cite{LiuWang}, we also let $a\preceq bx+c$ for any nonnegative $a,b,c$, and let $0\preceq f$ and $f\preceq 0$ for any $f\in \mathrm{PF}$.
A polynomial sequence $\{f_n(x)\}_{n\geq 0}$ with each $f_n(x)\in \mathrm{PF}$
is said to be a \emph{generalized Sturm sequence} if
$${f_0(x)\preceq f_1(x)\preceq \cdots \preceq  f_{n-1}(x)\preceq f_n(x)\preceq \cdots.}$$
Liu and Wang \cite{LiuWang} established the following sufficient condition for determining whether two polynomials have interlacing zeros.

\begin{theorem}\cite[Theorem 2.3]{LiuWang}\label{thm:Liu-Wang}
Let $F(x),f(x),g_1(x),\ldots,g_k(x)$ be real polynomials satisfying the following conditions.
\begin{itemize}
\item[$(i)$]There exist some real polynomials $\phi(x),\psi_1(x),\ldots,\psi_k(x)$ such that
\begin{align}\label{eq-rec-LW}
F(x)=\phi(x)f(x)+\psi_1(x)g_1(x)+\cdots+\psi_k(x)g_k(x),
\end{align}
and $\deg F(x)= \deg f(x)$ or $\deg F(x)= \deg f(x)+1$.
\item[$(ii)$] $f(x),\,g_j(x)$ are polynomials with only real zeros, and $g_j(x)\preceq f(x)$ for all $j$.
\item[$(iii)$]The leading coefficients of $F(x), g_1(x),\ldots,g_k(x)$ have the same sign.
\end{itemize}
Suppose that $\psi_j(r)\le 0$ for each $j$ and each zero $r$ of $f(x)$. Then $F(x)$ has only real zeros and $f(x)\preceq F(x)$.
In particular, if for each zero $r$ of $f(x)$, there exists an index $j$ such that
$g_j(x)\prec f(x)$  and $\psi_j(r)<0$, then $f(x)\prec F(x)$.
\end{theorem}

We are now in a position to present the interlacing property of $h_n(x;t)$, which is stronger than Theorem \ref{thm:rz-mfgnx}.

\begin{theorem}\label{thm:interlac-hnxt}
For any $t>0$ the sequence $\{h_n(x;t)\}_{n\ge 2}$ is a generalized Sturm sequence.
\end{theorem}

\begin{proof}
Fix $t>0$ and $n\ge 2$ throughout this proof. Denote $h_n(x):=h_n(x;t)$ for brief. We shall prove the desired result by using induction on $n$. First, by \eqref{eq:ex-g-poly-defi2} we have $h_2(x)=tx$ and $h_3(x)=tx^2+tx$. Clearly, both $h_2(x)$ and $h_3(x)$ are real-rooted polynomials in $x$, and
$$
 h_2(x)\preceq h_3(x).
$$
Assume that $h_n(x)$ and $h_{n+1}(x)$ are real-rooted polynomials and $h_n(x)\preceq h_{n+1}(x)$. It is sufficient to show
\begin{align}\label{intlc-n+12}
h_{n+1}(x)\preceq h_{n+2}(x).
\end{align}
To this end, we shall employ Theorem \ref{thm:Liu-Wang}. First, by recursion \eqref{rec:gfhx-n2}, we have
\begin{align*}
  h_{n+2}(x)&\,
=(1+x)h_{n+1}(x)+txh_n(x)
\end{align*}
which is of the form of \eqref{eq-rec-LW} with $k=1$, where
$$F(x)=h_{n+2}(x),\quad f(x)=h_{n+1}(x),\quad g_1(x)=h_{n}(x),$$ and
\begin{align*}
\phi(x)=1+x,\quad
\psi_1(x)=tx.
\end{align*}
By \eqref{eq:ex-g-poly-defi2}, we have that $\deg F(x)= \deg f(x)+1=n+1$. Thus, condition (i) of Theorem \ref{thm:Liu-Wang} is satisfied.
By inductive hypothesis that $h_n(x)\preceq h_{n+1}(x)$, we have condition (ii) of Theorem \ref{thm:Liu-Wang} holds, too.
Moreover, the leading coefficients of $F(x)$ and $g_1(x)$ are all positive by \eqref{eq:ex-g-poly-defi2}. Therefore, condition (iii) of Theorem \ref{thm:Liu-Wang} is also satisfied.
It remains to show that $\psi_1(r)\le 0$ for each zero $r$ of $f(x)$. Note that $f(x)$ has only nonpositive zeros since its coefficients are all nonnegative. Clearly, for $x\le 0$, we have $\psi_1(x)\le 0$.

It follows from Theorem \ref{thm:Liu-Wang} that \eqref{intlc-n+12} holds and $h_{n+2}(x)$ has only real zeros.
This completes the proof.
\end{proof}

\section{Proofs of main results}\label{sec:proof}

In this section, we complete the proofs of Theorems \ref{thm:g-poly-inflog} and \ref{thm:g-poly-hoTi}.

\begin{proof}[Proof of Theorem \ref{thm:g-poly-inflog}]
It is well-known that a real-rooted polynomial with nonnegative coefficients is log-concave. By Theorems \ref{thm:rz-mfgnx}, the polynomial $h_n(x;t)$ in $x$ has only real zeros for any $t>0$ and $n\ge 2$. It follows from Theorem \ref{thm:Branden} that the sequence $\{g_{U_{n,d}}(t)\}_{d=1}^{n-1}$ is infinitely log-concave for $n\ge 4$ and $t>0$.
\end{proof}

Ma\v{r}\'{i}k \cite{Marik1964} showed that if $f(x)=\sum_{k=0}^n \frac{a_k}{k!(n-k)!} x^k$ has only real zeros for $n\ge 3$, then the sequence $\{a_k\}_{k=0}^n$ satisfies the higher order Tur\'{a}n inequality \eqref{eq:ho-Turan}. Recently, Wang and Zhang \cite{Wang-Zhang2024} extended Ma\v{r}\'{i}k's result and established the following sufficient conditions for proving \eqref{eq:ho-Turan} for nonnegative sequences.

\begin{theorem}\cite[Theorem 1.2]{Wang-Zhang2024}\label{thm:Wang-Zhang}
Let $n\ge 3$ and $\{a_k\}_{k=0}^n$ be a sequence of positive integers. Let $f_1(x)=\sum_{k=0}^n a_k x^k$, $f_2(x)=\sum_{k=0}^n \frac{a_k}{k!} x^k$, and $f_3(x)=\sum_{k=0}^n \frac{a_k}{(n-k)!} x^k$. If one of $f_i(x)$, $1\le i \le 3$ has only real zeros, then the sequence $\{a_k\}_{k=0}^n$ satisfies the higher order Tur\'{a}n inequality \eqref{eq:ho-Turan}.
\end{theorem}

We are now ready to show a proof of Theorem \ref{thm:g-poly-hoTi}.

\begin{proof}[Proof of Theorem \ref{thm:g-poly-hoTi}]
Note that Theorem \ref{thm:Wang-Zhang} still holds for $a_k\in \mathbb{R}_{\ge 0}$, which can be verified by the argument in the proof of \cite[Theorem 1.2]{Wang-Zhang2024}. By Theorems \ref{thm:rz-mfgnx}, we have the real-rootedness of the polynomial $h_n(x;t)$ for any $t>0$ and $n\ge 2$. Set $a_k:=g_{n,k}(t)$ in Theorem \ref{thm:Wang-Zhang} with the convention that $a_0=a_n=0$. Then for any given $t>0$, we get the desired result that the sequence $\{g_{n,k}(t)\}_{k=1}^{n-1}$ satisfies the higher order Tur\'{a}n inequality \eqref{eq:ho-Turan} for $n\ge 5$.
This completes the proof.
\end{proof}

\section{Inequalities for $g_{n,d}(1)$}\label{sec:5}
As noted by Speyer \cite{Speyer2008, Speyer2009}, the number $S_i(n,d)=\binom{n-i-1}{d-i}\binom{n-d-1}{i-1}$ has interesting geometric interpretation. Therefore, $g_{n,d}(1)=\sum_{i=1}^d S_i(n,d)$ deserves more attention.

Let us recall some definitions and results showed by Speyer \cite{Speyer2008, Speyer2009}. Let $G(d,n)$ denote the Grassmannian of $d$-subspaces in $\mathbb{C}^n$.
Let $K=\bigcup_{n=1}^{\infty}\mathbb{C}((t^{1/n}))$ be the filed of Puiseux series, and let $v:K^* \rightarrow \mathbb{Q}$ be the map such that if $x=\sum_{i\ge M} a_i t^{i/N}$ with $a_M\neq 0$ then $v(x)=M/N$. Assume that there exists a $K$-valued point of $G(d,n)$ none of whose Pl\"{u}cker coordinate $p_I(x)$ are zero.
To helps to understand the combinatorial meaning of being a tropical Pl\"{u}cker vector, Speyer provided a simple polyhedral construction. Let $\Delta(d,n)$ denote the $(d,n)$-hypersimplex, the convex hull of the $\binom{n}{d}$ points $e_{i_1}+\cdots+e_{i_d}$ where $(i_1,\ldots,i_d)$ runs over $\binom{[n]}{d}$, with the convention that $[n]=\{1,2,\ldots,n\}$ and for any set $S$, $\binom{S}{d}$ is the set of $d$ element subsets of $S$. Let $I\mapsto P_I$ be a function $\binom{[n]}{d}\rightarrow \mathbb{Q}$.
Speyer \cite{Speyer2009} defined a polyhedral subdivision of $\Delta(d,n)$ as follows: Let $Q$ be the convex hull in $\Delta(d,n)\times \mathbb{R}$ of the points $(e_{i_1}+\cdots+e_{i_d},P_{i_1 \ldots i_d})$ and let $Q+\mathbb{R}_{\ge 0}$ denote the Minkowski sum of $Q$ with $\{0\}\times \mathbb{R}_{\ge 0}\subset \mathbb{R}^n \times \mathbb{R}$. Take the facets of $Q+\mathbb{R}_{\ge 0}$ whose outward pointing normal vectors have negative components in the last coordinate and project them down to $\Delta(d,n)$. This gives a polyhedral subdivision $\mathcal{D}_P$ of $\Delta(d,n)$ which is known as the regular subdivision associated to $P$. Then $P$ is a tropical Pl\"{u}cker vector if and only if, for each face $F$ of $\Delta$, the vertices of $F$ form the bases of a matroid. When this condition holds, the number of $i$-dimensional bounded faces of the corresponding tropical linear space is the number of $(n-i)$-dimensional interior faces in $\mathcal{D}$. Speyer obtained the following result and hence confirmed his $f$-vector conjecture \cite{Speyer2008} in the case of a tropical linear space realizable in characteristic zero.

\begin{theorem}\cite{Speyer2009}\label{thm:Speyer}
Suppose that $P$ is a tropical Pl\"{u}cker vector arising as $v(p_I(x))$ for some $x\in G(d,n)(K)$. Then $\mathcal{D}_P$ has at most $\binom{n-i-1}{d-i}\binom{n-d-1}{i-1}$ interior faces of dimension $n-i$, with equality if and only if all of the facets of $\mathcal{D}_P$ correspond to series-parallel matroids.
\end{theorem}

Theorem \ref{thm:Speyer} implies that $g_{n,d}(1)$ enumerates all the interior faces of $\mathcal{D}_P$ with dimension runs over $\{n-d,n-d+1,\ldots, n-1\}$ in the case that all of the facets of $\mathcal{D}_P$ correspond to series-parallel matroids. As a consequence of Theorems \ref{thm:g-poly-inflog} and \ref{thm:g-poly-hoTi}, we obtain the main result of this section.

\begin{theorem}\label{thm:inflg&hTq-gnd1}
For $n\ge 4$, the sequence $\{g_{n,d}(1)\}_{d=1}^{n-1}$ is infinitely log-concave. And for $n\ge 5$, the sequence $\{g_{n,d}(1)\}_{d=1}^{n-1}$ satisfies the higher order Tur\'{a}n inequality \eqref{eq:ho-Turan}.
\end{theorem}

Speyer \cite{Speyer2009} also gave an alternative description of Theorem \ref{thm:Speyer}. Let $T$ be the torus $(\mathbb{C}^*)^n/\textrm{diag} (\mathbb{C}^*)$ which acts on $G(d,n)$. Denote by $\overline{Tx}$ the closure of the torus orbit through $x$. Speyer constructed a certain subscheme $Y$ of $G(d,n)$, and mentioned that $Y$ is a union of torus orbit closures $Y=\bigcup_j \overline{Tx_j^1}$, one for each facet of $\mathcal{D}$. The $\overline{Tx_j^1}$ are glued along smaller torus orbit closures indexed by the smaller interior faces of $\mathcal{D}$. That is, the structure sheaf $\mathcal{O}_Y$ of $Y$ fits into an exact complex of sheaves
$$
0\rightarrow \mathcal{O}_Y \rightarrow \bigoplus_{j=1}^{f_1}\mathcal{O}_{\overline{Tx_j^1}}\rightarrow\cdots
\rightarrow \bigoplus_{j=1}^{f_i}\mathcal{O}_{\overline{Tx_j^i}}\rightarrow \cdots
\rightarrow \bigoplus_{j=1}^{f_n}\mathcal{O}_{\overline{Tx_j^n}}\rightarrow 0
$$
where each $x_j^i$ is a point of $G(d,n)$ such that the torus orbit $Tx_j^i$ is $(n-i)$-dimensional. Denote by $f_i$ the number of strata of dimension $n-i$. Then Theorem \ref{thm:Speyer} reads
$$
f_i\le \binom{n-i-1}{d-i}\binom{n-d-1}{i-1}
$$
with equality if and only if each $\overline{Tx_k^1}$ corresponds to a series-parallel matroid.

Clearly, one can see that if each $\overline{Tx_k^1}$ corresponds to a series-parallel matroid, then $g_{n,d}(1)=\sum_{i} f_i$ and the number $\sum_{i} f_i$ satisfies Theorem \ref{thm:inflg&hTq-gnd1}, too.

\section{The gamma-positivity}\label{sec:g-p}

In this section, we show the $\gamma$-positive of the polynomial $h_n(x;t)$.
Given a polynomial $f(x)=a_{r_1} x^{r_1}+\cdots + a_{r_s} x^{r_s} \in \mathbb{R}[x]$ with $r_1\le \cdots \le r_s$ and $a_{r_1}, a_{r_s}\neq 0$, the {\it darga} of $f(x)$ is defined to be ${\rm dar}(f):=r_1+r_s$, see Zeilberger \cite{Zeilberger1989}. The polynomial $f(x)$ is called {\it palindromic of darga} $n$ if $r_1+r_s=n$ and $a_{r_1+k}=a_{r_s-k}$ for all $k$. Let $\mathcal{P}_n$ denote the palindromic polynomials of darga $n$ with the convention that ${\rm dar}(0)$ can be any nonnegative integer. Sun, Wang and Zhang \cite[Theorem 2.8]{SWZ-2015} showed that $\mathcal{P}_n$ forms a linear space of dimension $\lfloor n/2 \rfloor +1$ and $\{x^i(1+x)^{n-2i}|i=0,1,\ldots,\lfloor n/2 \rfloor\}$ is a basis of $\mathcal{P}_n$.
Let $f(x)\in \mathcal{P}_n$ and
$$
f(x)=\sum_{i=0}^{\lfloor n/2 \rfloor} \gamma_i x^i (1+x)^{n-2i}.
$$
We say that $f(x)$ is {\it $\gamma$-positive of darga $n$} if all the $\gamma_i$ are nonnegative.

The notion of $\gamma$-positivity was first introduced by Fota and Sch\"{u}tzenberger \cite{Fota-Schutzenberger1970} in studying the classical Eulerian polynomials. The works of Stembridge \cite{Stembridge1997}, Br\"{a}nd\'{e}n \cite{Branden2004, Branden2008} and Gal \cite{Gal2005} made it clear that the concept of $\gamma$-positivity is of independently interest which provides a useful approach to the problem of unimodality for symmetric polynomials. As mentioned by Athanasiadis \cite[P. 2]{Athanasiadis}, symmetry and real-rootedness of a polynomial $f(x)\in\mathbb{R}_{\ge 0}[x]$ implies its $\gamma$-positivity, see also \cite[Lemma 4.1]{Branden2004} and \cite[Remark 3.1.1]{Gal2005}.

Clearly, the polynomial $h_n(x;t)=\sum_{d=1}^{n-1} g_{n,d}(t) x^d\in \mathcal{P}_n$ for any $t\in \mathbb{R}$, that is, $g_{n,1+k}(t)=g_{n,n-1-k}(t)$ for $0\le k\le n-1$ and $t\in \mathbb{R}$. It follows from Theorem \ref{thm:rz-mfgnx} and the above statement that $h_n(x;t)$ is $\gamma$-positive for $t>0$. More precisely, we have the following result.

\begin{theorem}\label{thm:gamma-p-gefgp}
Let $t>0$. For $n\ge 2$, we have
\begin{align}\label{eq:gm-p-hnx}
 h_n(x;t)=\sum_{i=0}^{\lfloor n/2 \rfloor} \gamma_{n,i} x^i (1+x)^{n-2i},
\end{align}
where $\gamma_{n,i}$ satisfies the following recurrence relation for $n\ge 4$ and $0\le i \le \lfloor n/2 \rfloor$,
\begin{align}\label{eq:rec-gamma-ni}
 \gamma_{n,i}=\gamma_{n-1,i}+t \,\gamma_{n-2,i-1}
\end{align}
with initial conditions $\gamma_{2,0}=\gamma_{3,0}=0$, $\gamma_{2,1}=\gamma_{3,1}=t$ and the conventions that $\gamma_{n,-1}=\gamma_{n,\lfloor n/2 \rfloor +1}=0$ for $n\ge 2$.
\end{theorem}

\begin{proof}
We shall prove Theorem \ref{thm:gamma-p-gefgp} by using induction on $n$. First by \eqref{eq:ex-g-poly-defi}, we have $h_2(x;t)=tx$, $h_3(x;t)=tx(1+x)$ and $h_4(x;t)=tx(1+x)^2+t^2x^2$. Clearly, \eqref{eq:gm-p-hnx} holds for $n=2, 3, 4$, and $\gamma_{2,0}=\gamma_{3,0}=0$, $\gamma_{2,1}=\gamma_{3,1}=t$. Moreover, it is easy to verify that the recurrence \eqref{eq:rec-gamma-ni} holds for $n=4$ and $0\le i\le 2$.
Assume \eqref{eq:gm-p-hnx} and \eqref{eq:rec-gamma-ni} hold true for some $n-1$ and $n$ where $n\ge 4$, we proceed to prove that
\begin{align}\label{eq:hn+1xt}
 h_{n+1}(x;t)=\sum_{i=0}^{\lfloor (n+1)/2 \rfloor} \big(\gamma_{n,i}+t \,\gamma_{n-1,i-1}\big) x^i (1+x)^{n+1-2i}.
\end{align}

By the recurrence \eqref{rec:gfhx-n2} given in Theorem \ref{rec:shiftn}, we have for some $n\ge 4$,
\begin{align*}
 h_{n+1}(x;t)&\,
=(1+x)h_{n}(x;t)+t x h_{n-1}(x;t).
\end{align*}
It follows from the inductive hypothesis that
\begin{align}
 h_{n+1}(x;t)
=&\, (1+x)\sum_{i=0}^{\lfloor n/2 \rfloor} \gamma_{n,i} x^i (1+x)^{n-2i}
  +t x \sum_{i=0}^{\lfloor (n-1)/2 \rfloor} \gamma_{n-1,i} x^i (1+x)^{n-1-2i}\nonumber\\
=&\, \sum_{i=0}^{\lfloor n/2 \rfloor} \gamma_{n,i} x^i (1+x)^{n+1-2i}
  +t \sum_{i=1}^{\lfloor (n+1)/2 \rfloor} \gamma_{n-1,i-1} x^{i} (1+x)^{n+1-2i}.\label{eq:hn+1xtrec}
\end{align}
Note that $\lfloor n/2 \rfloor=\lfloor (n+1)/2 \rfloor$ if $n$ is even and $\lfloor n/2 \rfloor +1=\lfloor (n+1)/2 \rfloor$ if $n$ is odd. In view of the convention that $\gamma_{n,-1}=\gamma_{n,\lfloor n/2 \rfloor +1}=0$ for $n\ge 2$, we get \eqref{eq:hn+1xt} from \eqref{eq:hn+1xtrec}.
This completes the proof.
\end{proof}

Clearly, Theorem \ref{thm:gamma-p-gefgp} implies that the polynomial $h_n(x;t)$ is $\gamma$-positivity of darga $n$ for any $t>0$ and $n\ge 2$. Besides, for $n\ge 2$,
\begin{equation*}
h_n(-1;t)=
\left\{
\begin{array}{ll}
0,\quad  & \textrm{ if } n \textrm{ is odd},\\
(-1)^{n/2} \gamma_{n,n/2}, \quad  & \textrm{ if } n \textrm{ is even}.
\end{array}
\right.
\end{equation*}
This means that $(-1)^{n/2} h_n(-1;t)\ge 0$ for $t>0$ and $n\ge 2$.

Notice that $\gamma_{n,i}$ is a function of $t$. For $t=1$, it reduces to $\binom{n-i-1}{i-1}$ with $0\le i\le \lfloor n/2 \rfloor$. The terms of $\{\binom{n-i-1}{i-1}\}_{i=1}^{\lfloor n/2 \rfloor}$ correspond to the sequence A011973 and A169803 in the OEIS founded by Sloane \cite{Sloane}.

\section{Asymptotic normality}\label{Sec:asynor}

This section is devoted to the study of asymptotic normality of Speyer's $g$-polynomial $g_{n,d}(t)$ for any given $t>0$. Theorem \ref{thm:rz-mfgnx} implies that the sequence $\{g_{n,d}(t)\}_{d=1}^{n-1}$ is log-concave and unimodal for any $t>0$.
The coefficients of many real-rooted polynomials appear to be asymptotically normal, see \cite{Bender, Canfield, CMW-2020, LWW-2023, Zhang-Zhao-2024} for examples. Motivated by this fact, we investigate that whether the generating function of Speyer's $g$-polynomials possess this asymptotic property.

Before showing the main result of this section, we first recall some concepts and results.
Let $\{f_n(x)\}_{n\geq 0}$ be a sequence of polynomials with nonnegative coefficients where
\begin{align}\label{def-f}
f_n(x)=\sum_{k=0}^{n}a(n,k)x^k.
\end{align}
The coefficient $a(n,k)$ is said to be asymptotically normal with mean $\mu_n$ and variance $\sigma_n^2$ by a central limit theorem if
\begin{align*}
\lim\limits_{n\rightarrow \infty} \sup\limits_{x\in \mathbb{R}}
\left| \sum\limits_{k\leq \mu_n+x\sigma_n} p(n,k) -\frac{1}{\sqrt{2\pi}}\int_{-\infty}^x \exp({-t^2/2}) dt \right|=0,
\end{align*}
where $p(n,k)=\frac{a(n,k)}{\sum_{j=0}^{n}a(n,j)}$.
The coefficient $a(n,k)$ is said to be asymptotically normal with mean $\mu_n$ and variance $\sigma_n^2$  by a local limit theorem on the real set $\mathbb{R}$ if
\begin{align*}
\lim\limits_{n\rightarrow \infty} \sup\limits_{x\in \mathbb{R}}
\left| \sigma_n p(n,\lfloor\mu_n+x \sigma_n\rfloor) -\frac{1}{\sqrt{2\pi}} \exp({-{x^2}/{2}}) \right|=0.
\end{align*}

The main result of this section is as follows.

\begin{theorem}\label{thm:asy-Sgp}
The coefficients of $h_n(x;t)$, that is, Speyer's $g$-polynomial $g_{n,d}(t)$, is asymptotically normal by local and central limit theorems for any $t>0$.
\end{theorem}

In order to prove Theorem \ref{thm:asy-Sgp}, we shall employ the following criterion, see Bender \cite{Bender} and Harper \cite{Harper}.

\begin{theorem}[{\cite[Theorem 2]{Bender}}]\label{lemm-asymp-normal}
Let $\{f_n(x)\}_{n\geq 0}$ be a real-rooted polynomial sequence with nonnegative coefficients as in \eqref{def-f}.
Let
\begin{align}
\mu_n=\frac{f_n'(1)}{f_n(1)}\quad \textrm{and} \quad
\sigma_n^2=\frac{f_n''(1)}{f_n(1)}+\mu_n-\mu_n^2.\label{eq-mu-var}
\end{align}
If $\sigma_n^2\rightarrow +\infty$ as $n \rightarrow +\infty$, then the coefficient of $f_n(x)$ is asymptotically normal  with mean $\mu_n$ and variance $\sigma_n^2$  by local and central limit theorems.
\end{theorem}

We are now in a position to show a proof of Theorem \ref{thm:asy-Sgp}.

\begin{proof}[Proof of Theorem \ref{thm:asy-Sgp}]
Fix $t>0$ throughout this proof. For $n\ge 2$, let $f_n(x):=h_n(x;t)$, $\mu_n=\frac{f_n'(1)}{f_n(1)}$ and $\sigma_n^2=\frac{f_n''(1)}{f_n(1)}+\mu_n-\mu_n^2$. By Theorem \ref{thm:rz-mfgnx}, we have $\{f_n(x)\}_{n\ge 2}$ is a real-rooted polynomial sequence. Clearly, the coefficients of $f_n(x)$ are nonnegative by \eqref{g-uni-dn(t)}. By Theorem \ref{lemm-asymp-normal} it is sufficient to show that $\sigma_n^2\rightarrow +\infty$ as $n \rightarrow +\infty$. To this end, we need first obtain the limitation of ${f_{n+1}(1)}/{f_{n}(1)}$.
Substitute $x=1$ in the recurrence \eqref{rec:gfhx-n2}, we have
\begin{align}\label{eq:rec:fn21n1}
 f_{n+2}(1)&\,
=2 f_{n+1}(1)+t f_n(1),\qquad n\ge 2.
\end{align}
Let $r_n(x)=\frac{f_{n+1}(x)}{f_{n}(x)}$ for $n\ge 2$. Suppose that
\begin{align*}
 \lim_{n\rightarrow +\infty} r_n(1)=\lambda,
\end{align*}
where $\lambda$ is a finite real number. It follows from \eqref{eq:rec:fn21n1} that
\begin{align*}
 \lambda^2 = 2 \lambda+ t,
\end{align*}
whose zeros are
$1\pm\sqrt{t+1}$.
Since $r_n(1)>0$ for all $n\ge 2$, we have
\begin{align}\label{lim:Afn/n-1}
 \lambda=\lim_{n\rightarrow +\infty} \frac{f_{n+1}(1)}{f_{n}(1)}=1+\sqrt{t+1}.
\end{align}

We proceed to compute $\mu_n$.
By setting $x=1$ in the recursions \eqref{rec:gfhx-nx'}, we have
\begin{align*}
0\cdot f_{n+1}(1)&\,
=-2n(t+1) f_{n}(1)+4(t+1) f_{n}'(1),
\end{align*}
and hence,
\begin{align}\label{eq:fn'1/fn1}
 \mu_n=\frac{f_{n}'(1)}{f_{n}(1)}
=\frac{2n(t+1)}{4(t+1)}=\frac{n}{2}.
\end{align}

Letting $x=1$ in \eqref{rec:gfhx-nx''}, we get
\begin{align*}
 4(n-1)(t+1) f_{n+1}(1)&\,
=-4n(t+1)(nt+n-2t-3) f_{n}(1)+(4t+4)^2 f_n''(1),
\end{align*}
Thus,
\begin{align}\label{eq:fn''1/fn1}
 \frac{f_n''(1)}{f_n(1)}
=\frac{n-1}{4(t+1)}\frac{f_{n+1}(1)}{f_{n}(1)}+\frac{n(nt+n-2t-3)}{4(t+1)}.
\end{align}

It remains to compute $\sigma_n^2$ and its limitation as $n$ goes to $+\infty$.
By \eqref{eq:fn'1/fn1} and \eqref{eq:fn''1/fn1}, we derived that
\begin{align*}
 \sigma_n^2
=&\, \frac{f_n''(1)}{f_n(1)}+\mu_n-\mu_n^2\\
=&\, \frac{n-1}{4(t+1)}\frac{f_{n+1}(1)}{f_{n}(1)}+\frac{n(nt+n-2t-3)}{4(t+1)}+\frac{n}{2}-\frac{n^2}{4}\\
=&\, \frac{n-1}{4(t+1)}\frac{f_{n+1}(1)}{f_{n}(1)}-\frac{n}{4(t+1)}.
\end{align*}
By \eqref{lim:Afn/n-1},
\begin{align*}
 \lim_{n\rightarrow +\infty}\frac{\sigma_n^2}{n}
=&\, \lim_{n\rightarrow +\infty} \frac{n-1}{4n(t+1)}\frac{f_{n+1}(1)}{f_{n}(1)}
   -\lim_{n\rightarrow +\infty}\frac{n}{4n(t+1)}\\
=&\, \frac{1}{4(t+1)}\big(1+\sqrt{t+1}\big)-\frac{1}{4(t+1)}\\
=&\, \frac{1}{4\sqrt{t+1}}.
\end{align*}
So, we have
$
 \sigma_n^2 \sim \frac{n}{4\sqrt{t+1}}.
$
and hence for $t>0$,
$
 \lim_{n\rightarrow +\infty}\sigma_n^2
=+\infty.
$
This completes the proof.
\end{proof}

\section{Laguerre inequality}\label{sec:L}
In this section, we show the Laguerre inequalities arising from Speyer's $g$-polynomial $g_{n,d}(t)$ and its generating function $h_n(x;t)$ for positive $t$. A polynomial $f(x)$ is said to satisfies the \emph{Laguerre inequality} if
\begin{align}\label{defi:L-ineq}
 f'(x)^2-f(x)f''(x)\ge 0.
\end{align}
It was stated by Laguerre \cite{Laguerre1989} that if $f(x)$ is a real-rooted polynomial, then the Laguerre inequality \eqref{defi:L-ineq} holds for $f(x)$, see also \cite{Dou-Wang-2023}.

Then we obtain the following result from Theorem \ref{thm:rz-mfgnx}.

\begin{theorem}
Given $t>0$, the polynomial $f(x):=h_n(x;t)$ in $x$ satisfies the Laguerre inequality \eqref{defi:L-ineq} for $n\ge 2$.
\end{theorem}

It was proved in \cite[Theorem 3.1]{Zhang-Zhao-2024} that Speyer's $g$-polynomial $g_{n,d}(x)$ has only real zeros. Hence, we get the following consequence.
\begin{theorem}
The Speyer's $g$-polynomial $g_{n,d}(x)$ satisfies the Laguerre inequality \eqref{defi:L-ineq} for $n\ge 2$ and $1\le d \le n-1$.
\end{theorem}

Jensen \cite{Jensen} introduced a definition of generalized Laguerre inequality, that is, for $f(x)\in \mathcal{L}{\textrm{-}}\mathcal{P}$,
\begin{align}\label{defi:L-ineq-ge}
L_r(f(x))=\frac{1}{2}\sum_{k=0}^{2r}(-1)^{r+k}\binom{2r}{k}f^{(k)}(x) f^{(2r-k)}(x)\ge 0,
\end{align}
where $f^{(k)}(x)$ is the $k$-th derivative of $f(x)$. Clearly, \eqref{defi:L-ineq-ge} reduces to the classical Laguerre inequality \eqref{defi:L-ineq} for $r=1$. It is worth mentioning that Csordas and Vishnyakova \cite{Csordas-Vishnyakova} showed that if a function $f(x)$ satisfies \eqref{defi:L-ineq-ge} for all $r$ and all $x\in \mathbb{R}$, then $f(x)\in \mathcal{L}{\textrm{-}}\mathcal{P}$, which implies that the Laguerre inequality is a characterizing property of functions belonging to the Laguerre-P\'{o}lya class.
For more properties on Laguerre inequality, see \cite{Wagner2022, Dou-Wang-2023, Wang-Yang2024} and the references therein.

To conclude this section, we propose the following conjecture.

\begin{conjecture}
Given $t>0$, the polynomial $h_n(x;t)$ in $x$ satisfies the generalized Laguerre inequality \eqref{defi:L-ineq-ge} for $r\ge 1$.
\end{conjecture}

\section*{Acknowledgments}
The author wishes to thank Arthur L. B. Yang for introducing him to this topic and for valuable discussions.

\end{document}